\documentclass[11pt]{amsart}
\usepackage{amssymb,amsmath,amsthm}
\usepackage{amsfonts}
\usepackage{enumerate}
\usepackage{dsfont}
\usepackage{cite}
\usepackage[colorlinks, citecolor=red]{hyperref}
\usepackage{mathrsfs}
\usepackage{epsfig}
\usepackage{lscape}
\usepackage{subfigure}
\usepackage{epstopdf}
\usepackage{caption}
\usepackage{algorithm}
\usepackage{algpseudocode}
\usepackage{multirow}
\usepackage{geometry}
\usepackage{paralist}
\usepackage{enumerate}
\usepackage{url}
\usepackage{graphicx,cite}
\usepackage{longtable}
\usepackage{tikz}

\renewcommand{\t}{^\top}
\renewcommand{\k}{^{(k)}}
\newcommand{\kp}{^{(k+1)}}
\newcommand{\km}{^{(k-1)}}
\newcommand{\zo}{^{(0)}}
\newcommand{\ol}{^{(1)}}
\renewcommand{\i}{_i}
\newcommand{\ik}{_{i_k}}

\newcommand{\Sk}{\mathcal{S}_k}

\newcommand{\Ik}{\mathcal{I}_k}
\newcommand{\Jk}{\mathcal{J}_k}
\newcommand{\s}{^{\star}}
\newcommand{\nm}[1]{\| #1 \|_2}

\newcommand{\nmsq}[1]{\| #1 \|_2^2}
\newcommand{\nmfsq}[1]{\| #1 \|_F^2}

\newcommand{\inn}[1]{\langle #1 \rangle}

\newcommand{\prob}{\operatorname{Prob}}

\newcommand{\Rn}{\mathbb{R}^n}
\newcommand{\Rmn}{\mathbb{R}^{m \times n}}
\newcommand{\E}{\mathbb{E}}

\newcommand{\al}{\alpha}
\newcommand{\be}{\beta}
\newcommand{\hf}{\frac{1}{2}}
\newcommand{\sima}{\sigma_{\max}}
\newcommand{\simi}{\sigma_{\min}}

\newcommand{\setm}{\{1,\cdots,m\}}
\newcommand{\seqk}[1]{\{ #1 \}_{k\geq0}}

\newcommand{\termaa}{\max_{ i \in[m] } \left\{ \frac{|\langle a_i, x\k\rangle- b\i|^2}{\nmsq{a\i}} \right\}}
\newcommand{\termbb}{\frac{\nmsq{Ax\k-b}}{\nmfsq{A}}}
\newcommand{\termcc}{\frac{|\langle a_i, x\k\rangle- b\i|^2}{\nmsq{a\i}}}

\newcommand{\RA}{\operatorname{Range}(A\t)}

\newtheorem{definition}{Definition}[section]
\newtheorem{corollary}[definition]{Corollary}

\newtheorem{theorem}[definition]{Theorem}
\newtheorem{lemma}[definition]{Lemma}

\newtheorem{remark}[definition]{Remark}

\newcommand{\res}[2]{\inn{a#1,#2}-b#1}

\date{}

\begin{document}
	
\title
[On the convergence analysis of GRK]{On the convergence analysis of  the greedy randomized Kaczmarz method}
	
\author{Yansheng Su}
	\address{School of Mathematical Sciences, Beihang University, Beijing, 100191, China. }
	\email{suyansheng@buaa.edu.cn}

	\author{Deren Han}
	\address{LMIB of the Ministry of Education, School of Mathematical Sciences, Beihang University, Beijing, 100191, China. }
	\email{handr@buaa.edu.cn}
	
	\author{Yun Zeng}
	\address{School of Mathematical Sciences, Beihang University, Beijing, 100191, China. }
	\email{zengyun@buaa.edu.cn}
	
	\author{Jiaxin Xie}
	\address{LMIB of the Ministry of Education, School of Mathematical Sciences, Beihang University, Beijing, 100191, China. }
	\email{xiejx@buaa.edu.cn}
		
	\maketitle
	
	\begin{abstract}
		In this paper, we analyze the greedy randomized Kaczmarz (GRK) method proposed in Bai and Wu (SIAM J. Sci. Comput., 40(1):A592--A606, 2018)  for solving linear systems. We develop more precise greedy probability criteria to effectively select the working row from the coefficient matrix. Notably, we prove that the linear convergence of the GRK method is deterministic and demonstrate that using a tighter threshold parameter can lead to a faster convergence rate. Our result revises existing convergence analyses, which are solely based on the expected error by realizing that the iterates of the GRK method are random variables. Consequently, we obtain an improved iteration complexity for the GRK method.
Moreover,  the Polyak's heavy ball momentum technique is incorporated to improve the performance of the GRK method. We propose a refined convergence analysis, compared with the technique used in Loizou and Richt\'{a}rik (Comput. Optim. Appl., 77(3):653--710, 2020), of momentum variants of randomized iterative methods, which shows that the proposed GRK method with momentum (mGRK) also enjoys a deterministic linear convergence. Numerical experiments show that the mGRK method is more efficient than the GRK method.
	\end{abstract}
	\let\thefootnote\relax

\let\thefootnote\relax

\section{Introduction}
	

The Kaczmarz method \cite{kaczmarz1937angenaherte}, also known as \emph{algebraic reconstruction technique} (ART) \cite{herman1993algebraic,gordon1970algebraic}, is an iterative method for solving large-scale linear systems
	\begin{equation}\label{equ:problem}
		Ax = b, \ A \in \Rmn, b\in \mathbb{R}^n.
	\end{equation}
Throughout this paper, we assume that the linear system \eqref{equ:problem} is consistent, i.e. there exists a $ x\s $ such that $ Ax\s = b $. For any $i\in[m]:=\{1,\ldots,m\}$, we  use $ a\i $ to denote the transpose of the $i$-th row of $ A $ and use $b_i$ to denote the $i$-th  entry of $b$. We further assume $ \nm{a_i} \neq 0 $ for all $ i \in [m].$ Starting from  $x^{(0)}\in\mathbb{R}^n$, the original Kaczmarz method constructs $x\kp$ by
	\begin{equation}\label{equ:basiciter}
		 x\kp = x\k - \frac{\langle a\ik, x\k\rangle - b\ik}{\nmsq{a\ik}} a\ik,
	\end{equation}
where the index $ i_k = (k \mod m) +1$.
There are empirical evidences that selecting a working row from the matrix $A$ randomly can often lead to a better convergence of the Kaczmarz method compared to choosing it sequentially \cite{herman1993algebraic,natterer2001mathematics,feichtinger1992new,sun2021worstcase}.
The celebrated result of Strohmer and Vershynin \cite{strohmer2009randomized} shows that if the index $ i_k$ is selected randomly with probability proportional to $\|a_{i_k}\|^2_2$, then the resulting \emph{randomized  Kaczmarz} (RK) method converges linearly in expectation.
Since then, RK-type methods have received extensive attention due to their computational efficiency and scalability. They have a wide range of applications in many areas of scientific computing and engineering such as computerized tomography \cite{hansen2021computed}, signal processing \cite{Byr04}, optimal control \cite{Pat17}, and machine learning \cite{Cha08}. We refer to \cite{bai2023randomized,bai2023convergence} for a comprehensive survey  on the Kaczmarz method.

	The RK method has an evident weakness in its probability criterion for selecting working rows, which can result in slow convergence if some rows have significantly larger Euclidean norms compared to others.
	To address these limitations and improve the convergence of the RK method, several variations have been proposed, such as the accelerated  randomized Kaczmarz method \cite{liu2015accelerated,zeng2023adaptive}, the greedy Motzkin--Kaczmarz method \cite{zhang2022greedy}, and the  weighted  randomized Kaczmarz method \cite{Ste20Wei}.  In \cite{bai2018greedy}, Bai and Wu
 introduced the \emph{greedy probability criterion}, and constructed the
\emph{greedy randomized Kaczmarz} (GRK) method for solving the linear system \eqref{equ:problem}.
At the $k$-th iteration,  GRK determines a subset $\mathcal{I}_k$ of $[m]$ such that the magnitude of the residual $\langle a_{i}, x\k\rangle-b_i$ exceeds a threshold,
i.e.
\begin{equation}\label{equ:Ikorigrd}
	\Ik = \left\{  i : \frac{|\langle a\i, x\k\rangle- b\i|^2}{\nmsq{a\i}}  \geq \frac{1}{2} \left( \max_{i\in[m]} \left\{ \frac{|\langle a\i, x\k\rangle- b\i|^2}{\nmsq{a\i}} \right\} +\frac{\nmsq{Ax\k-b}}{\nmfsq{A}} \right)   \right\}.
\end{equation}
With the introduction of a modified residual vector defined as
$$
\tilde{r}\k_i= \begin{cases}\langle a_{i}, x\k\rangle-b_i, & \text { if } i \in \mathcal{I}_k, \\ 0, & \text { otherwise},\end{cases}
$$
the GRK method selects the index of the working row $i_k \in \mathcal{I}_k$ with  probability
\begin{equation}\label{xie-230624}
\operatorname{Prob}\left(i_k=i\right)=\frac{|\tilde{r}\k_{i}|^2}{\|\tilde{r}\k\|_2^2} .
\end{equation}
Finally, the GRK method orthogonally projects the current iterate $x^k$ onto the $i_k$ hyperplane $\{x\mid \langle a_{i_k},x\rangle=b\ik\}$ to obtain the next iterate $x^{k+1}$.
By using the above greedy approach, rows that are related to small entries of the residual vector $Ax^k-b$ may not be selected, which guarantees the efficiency of each iteration of GRK. This property leads to a faster convergence rate of GRK compared to RK.
In recent years,  there  has been a large amount of work on the refinements and extensions of the GRK method, such as the capped adaptive  sampling rule \cite{gower2021adaptive,yuan2022adaptively}, the greedy augmented randomized Kaczmarz method for inconsistent linear systems \cite{bai2021greedy}, the greedy randomized coordinate descent method \cite{bai2019greedy}, and the capped nonlinear Kaczmarz method \cite{zhang2022greedy}. In addition, we note that there has also been some work on non-random Kaczmarz methods \cite{chen2022fast,niu2020greedy,zhang2019new} inspired by the GRK method.



In this paper, we show that the linear convergence of the GRK method is deterministic.
 In general, the convergence analyses of RK-type methods are related to the expectation of the error $ \E \left[\nmsq{x\k -x\s} \right]$.
Specifically, Bai and Wu \cite{bai2018greedy} proved that the iteration sequence of GRK satisfies
	\begin{equation}\label{equ:convforgrdintexp}
		\E [\nmsq{x\k-x\s}] \leq \left(  1-\hf\left( \frac{1}{\gamma}\nmfsq{A}+1 \right) \frac{\simi^2(A)}{\nmfsq{A}} \right)^{k-1} \left(1-\frac{\simi^2(A)}{\nmfsq{A}}\right) \nmsq{x\zo-x\s},		
	\end{equation}
where $ \gamma = \max_{ i \in [m]} \sum_{j=1, j\neq i}^{m} \nmsq{a_j} < \nmfsq{A}$. 
We demonstrate that the greedy  strategy will always guarantee a certain reduction of $ \nmsq{x\k -x\s} $ at each iteration. As a result, the convergence bound in \eqref{equ:convforgrdintexp} is not only valid for the expectation $ \E \left[\nmsq{x\k -x\s} \right]$, but also for the quantity of $\nmsq{x\k -x\s}$. Particularly, we show that the iteration sequence of GRK satisfies
$$		\nmsq{x\k-x\s} \leq \bigg(1-\frac{\simi^2(A)}{\gamma}\bigg)^{k-1} \bigg(1-\frac{\simi^2(A)}{\nmfsq{A}}\bigg) \nmsq{x\zo-x\s},
$$
where the convergence rate is improved by a slightly adjusted row selection criterion (See Corollary \ref{main-cor}). To the best of our knowledge, this is the first time that convergence not related to the expectation has been explored for RK-derived methods. We refer to such convergence for random algorithms as \textit{deterministic}. Furthermore, our deterministic linear convergence result ensures that the GRK method can achieve a better iteration complexity than the one obtained from the expected error $\mathbb{E}[\|x^{(k)}-x\s\|^2_2]$.


Our results are primarily based on the following observation:  By the definition of $ \Ik $, we know that any index $ i \in\Ik $ satisfies
$$
		\termcc \geq  \frac{1}{2}\left( \termaa +   \termbb\right),
$$
and	since
$$\termaa \geq \sum_{i=1}^{m} \frac{\nmsq{a\i}}{\nmfsq{A}} \cdot \termcc = \termbb,$$
for any $i \in \Ik$, the following inequality always holds regardless of the probability employed
	\begin{equation}\label{equ:termrela}
		\termcc \geq \termbb,
	\end{equation}
which ensures deterministic convergence. In fact, this paper will introduce a tighter threshold parameter to determine the indexes that belong to the set $\Ik$ (See \eqref{alg-xie-1}). As a result, we obtain a tighter lower bound than \eqref{equ:termrela} (See \eqref{key-obs} and \eqref{gamma-est}).

Furthermore,  our observation makes it possible for many variants of the GRK method to achieve deterministic convergence. Particularly, we investigate the heavy ball momentum \cite{polyak1964some} variant of the GRK method and demonstrate that the linear convergence of the GRK method with momentum (mGRK) is also deterministic. In fact, the incorporation of momentum acceleration techniques with Kaczmarz-type methods has been a popular topic in the literature \cite{loizou2020momentum,liu2015accelerated,morshed2020accelerated,morshed2022sampling, zeng2023adaptive,han2022pseudoinverse,HSXDR2022}. Convergence of the heavy ball momentum variant of the randomized Kaczmarz method  has previously been analyzed in \cite{loizou2020momentum} and \cite{han2022pseudoinverse}. Additionally, we provide an alternative convergence analysis of the heavy ball momentum variant of the Kaczmarz method, where a smaller convergence factor for the Kaczmarz method with momentum can be obtained.

\subsection{Notations}
We here give some notations that will be used in the paper. For vector $x\in\mathbb{R}^n$, we use $x_i,x\t$ and $\|x\|_2$ to denote the $i$-th entry, the transpose and the Euclidean norm of $x$, respectively. For matrix $A\in\mathbb{R}^{m\times n}$, we use $a_i$, $A\t$, $A^\dagger$, $ \|A\|_F $, $ \operatorname{Range}(A) $, $\text{Rank}(A)$, $ \sima(A) $, and $ \simi(A) $ to denote the $i$-th row, the transpose, the Moore-Penrose pseudoinverse, the Frobenius norm, the range space, the rank, the largest and the smallest non-zero singular value of $ A $, respectively. Given $\mathcal{S} \subseteq [m]:=\{1,\ldots,m\}$, the cardinality of the set $\mathcal{S} $ is denoted by $|\mathcal{S} |$ and the complementary set of $\mathcal{S} $ is denoted by $\mathcal{S} ^c$, i.e. $\mathcal{S} ^c=[m]\setminus S$.
		
\subsection{Organization}
The remainder of the paper is organized as follows. We prove the deterministic convergence of GRK in Section 2. In Section 3, we propose the momentum variant of the GRK method. In Section 4, we perform some numerical experiments to show the effectiveness of the proposed method. Finally, we conclude the paper in Section 5.

\section{Improved Greedy randomized Kaczmarz method }

In this section, we investigate the deterministic convergence of the greedy randomized Kaczmarz (GRK) method and introduce a precise probability criterion to enhance its performance.  We propose the improved GRK (iGRK) method, which is outlined in Algorithm \ref{algo:basic}. We note that, compared to the GRK method proposed by Bai and Wu \cite{bai2018relaxed} where the probability criterion \eqref{xie-230624} is used, any probability that satisfies
$$
\left\{
	\begin{matrix}
		\ 	p\k_i
 = 0, &  i \notin \Jk,\\
		\ 	p\k_i
 \geq 0, &  i \in \Jk,
	\end{matrix}
	\right. \quad  \mathrm{and} \quad  \sum_{i\in\Jk} p\k_i=1,
$$
will be appropriate for Algorithm \ref{algo:basic} \cite{miao2022greedy}. Moreover, the row selection criterion is slightly adjusted compared to \eqref{equ:Ikorigrd} for an improved convergence rate.

\begin{algorithm}[htpb]
\caption{ Improved Greedy randomized Kaczmarz (iGRK) method \label{algo:basic}}
\begin{algorithmic}
\Require
 $A\in \mathbb{R}^{m\times n}$, $b\in \mathbb{R}^m$, $k=0$ and initial point $x\zo \in\mathbb{R}^n$.
\begin{enumerate}
\item[1:] Determine the set
\[
\mathcal{N}_k = \{ i\in[m]: |\inn{a_i,x\k}-b_i|\neq0 \}
\]
and compute \begin{equation}\label{equ:Gamma}
	\Gamma_k = \sum_{i\in\mathcal{N}_k} \nmsq{a_i}.
\end{equation}

\item[2:] Determine the index set of positive integers
\begin{equation}\label{alg-xie-1}
\Jk = \left\{  i : \frac{|\langle a\i, x\k\rangle- b\i|^2}{\nmsq{a\i}}  \geq \frac{1}{2} \bigg( \max_{i\in[m]} \left\{ \frac{|\langle a\i, x\k\rangle- b\i|^2}{\nmsq{a\i}} \right\} +\frac{\nmsq{Ax\k-b}}{\Gamma_k} \bigg)   \right\}.
\end{equation}
\item[3:] Select $ i_k \in  \Jk $ according to some probability
$$
			\prob(i_k = i) = p\k_i,
$$			
		where $ p\k_i = 0 \operatorname{if} i \notin \Jk $, $ p\k_i \geq 0 \operatorname{if} i \in \Jk $, and $ \sum_{i\in\Jk} p\k_i=1 $.
\item[4:] Set
$$
x\kp = x\k -  \frac{\langle a\ik, x\k\rangle - b\ik}{\nmsq{a\ik}}a\ik. 	
$$
\item[5:] If the stopping rule is satisfied, stop and go to output. Otherwise, set $k=k+1$ and return to Step $1$.
\end{enumerate}

\Ensure
  The approximate solution $x^k$.
\end{algorithmic}
\end{algorithm}

The iGRK method is well-defined as the index set $ \Jk $ is nonempty. In fact, for any $k\geq0$, we have
\begin{equation}\label{xie-prof-113}
	\begin{aligned}
		\max_{i\in[m]} \left\{ \frac{|\langle a\i, x\k\rangle- b\i|^2}{\nmsq{a\i}} \right\} &\geq  \sum_{i\in\mathcal{N}_k} \frac{\nmsq{a_i}}{\Gamma_k} \frac{|\langle a\i, x\k\rangle- b\i|^2}{\nmsq{a\i}} \\
		&=  \frac{\sum_{i\in\mathcal{N}_k}|\langle a\i, x\k\rangle- b\i|^2}{\Gamma_k} 
		\\&= \frac{\sum_{i=1}^m|\langle a\i, x\k\rangle- b\i|^2}{\Gamma_k} \\
		&= \frac{\nmsq{Ax\k-b}}{\Gamma_k},
	\end{aligned}
\end{equation}
where the inequality follows from the fact that $ \max_{i\in[m]} \left\{ \frac{|\langle a\i, x\k\rangle- b\i|^2}{\nmsq{a\i}} \right\} $ is the largest weighted average of $ \frac{|\langle a\i, x\k\rangle- b\i|^2}{\nmsq{a\i}}, i\in[m]$ and the second equality follows from $ \sum_{i\in\mathcal{N}_k^c}|\langle a\i, x\k\rangle- b\i|^2 =0$. It follows from \eqref{xie-prof-113} that at least $ i_k^{\max}$ belongs to $ \Jk $, where $ i_k^{\max} \in \arg\max\limits_{i\in[m]} \left\{  \frac{|\res{\i}{x\k}|^2}{\nmsq{a\i}} \right\} $. Thus $ \Jk $ is nonempty and  the iGRK method is well-defined.

\begin{remark}
	By setting $\Gamma_k$ in  \eqref{equ:Gamma} as $\Gamma_k=\|A\|^2_F$ for all $k\geq 0$ and using \eqref{xie-230624} as the probability criterion, Algorithm \ref{algo:basic} effectively implements the original GRK method. 
	In fact, the modification in \eqref{equ:Gamma} is intended to provide a tighter threshold parameter, as $ \Gamma_k< \nmfsq{A}$ for $k\geq 1$ $($as shown in  \eqref{gamma-est}$)$. As a result, we can get a better convergence rate $($see Remark \ref{remark-xie-114-1}$)$.
\end{remark}

\subsection{Deterministic convergence}
The convergence result for Algorithm \ref{algo:basic} is as follows. 
%
%
\begin{theorem}\label{theo:basic}
Suppose that $ x\zo \in \Rn $ and let $ x\s = A^\dagger b + (I-A^\dagger A)x\zo $ denote the projection of $ x\zo $ onto the solution set of $ Ax=b$. For $ k \geq 0 $, the iteration sequence $ \seqk{x\k} $ generated by Algorithm \ref{algo:basic} satisfies
\[
\nmsq{x\kp-x\s} \leq \left(1-\frac{\simi^2(A)}{\Gamma_k}\right) \nmsq{x\k-x\s},
\]
where $ \Gamma_k $ is defined as \eqref{equ:Gamma}.
\end{theorem}	
\begin{proof}
 According to \eqref{xie-prof-113} and the definition of $ \Jk $, for any $ i \in \Jk $ we have
\begin{equation}\label{key-obs}
\begin{aligned}
\frac{|\langle a\i, x\k\rangle- b\i|^2}{\nmsq{a\i}} \geq \frac{1}{2} \bigg( \max_{i\in[m]} \left\{ \frac{|\langle a\i, x\k\rangle- b\i|^2}{\nmsq{a\i}} \right\} +\frac{\nmsq{Ax\k-b}}{\Gamma_k} \bigg) \geq \frac{\nmsq{Ax\k-b}}{\Gamma_k},
\end{aligned}
\end{equation}
Then by the iterative strategy of Algorithm \ref{algo:basic}, we have
\begin{equation}\label{equ:basicprof2}
\begin{aligned}
\nmsq{x\kp - x\s} = & \left\|x\k -  \frac{\langle a\ik, x\k\rangle - b\ik}{\nmsq{a\ik}}a\ik-x\s \right\|^2_2\\
= & \nmsq{x\k-x\s} - \frac{|\langle a\ik, x\k\rangle - b\ik|^2}{\nmsq{a\ik}} \\
\leq & \nmsq{x\k-x\s} - \frac{\nmsq{Ax\k-b}}{\Gamma_k},
\end{aligned}
\end{equation}
where the second equality follows from the fact that $b\ik=\langle a\ik, x\s\rangle$. 
Next, we give an estimate for $\nmsq{Ax\k -b}$. First, we show that for any $k\geq0$, $ x\k - x\s \in \RA $. Indeed, from the definition of $ x\s $, we know that $ x\zo - x\s = A^\dagger(Ax\zo - b) \in \RA $. Suppose that $ x\k - x\s \in \RA $ holds, then $ x\kp - x\s = x\k -x\s -  \frac{\langle a\ik, x\k\rangle - b\ik}{\nmsq{a\ik}}a\ik \in \RA $. By induction we have that  $ x\k - x\s \in \RA $ holds for any $ k\geq 0. $ Therefore,
$$
\nmsq{Ax\k -b} = \nmsq{A(x\k-x\s)} \geq \simi^2(A) \nmsq{x\k-x\s}.
$$
Substituting it into \eqref{equ:basicprof2} completes the proof.
\end{proof}


We note that \eqref{key-obs} indicates that  
 the set $ \Jk $ in Algorithm \ref{algo:basic} actually consists of residuals larger than a certain threshold. 
Therefore, instead of taking the expectation of $ \frac{|\langle a\ik, x\k\rangle - b\ik|^2}{\nmsq{a\ik}} $, we employ a uniform bound of it, which makes the result deterministic. Since $\Gamma_k\leq \|A\|^2_F$, we know that \eqref{key-obs} implies \eqref{equ:termrela}. In fact, for any $k\geq 1$, we can derive a more tighter upper bound for $\Gamma_k$.

\begin{corollary}\label{main-cor}
Suppose that $ x\zo \in \Rn $ and let $ x\s = A^\dagger b + (I-A^\dagger A)x\zo $ denote the projection of $ x\zo $ onto the solution set of $ Ax=b$. Then for $ k \geq 1 $, the iteration sequence $ \seqk{x\k} $ generated by Algorithm \ref{algo:basic} satisfies
$$
\nmsq{x\k-x\s} \leq \left( 1- \frac{\simi^2(A)}{\gamma}\right) ^{k-1} \left(1-\frac{\simi^2(A)}{\nmfsq{A}}\right) \nmsq{x\zo-x\s},
$$
where $ \gamma = \max_{i\in[m]} \sum_{j=1, j\neq i}^{m} \nmsq{a_j}$.
\end{corollary}
\begin{proof}
	 For $ k \geq 1 $, we have
	\[\begin{aligned}
		r\k_{i_{k-1}} = &  \langle a_{i_{k-1}}, x\k\rangle - b_{i_{k-1}}\\
		=&  \left\langle a_{i_{k-1}}, x\km - \frac{\langle a_{i_{k-1}}, x\km\rangle - b_{i_{k-1}}}{\nmsq{a_{i_{k-1}}}}a_{i_{k-1}}\right\rangle - b_{i_{k-1}} \\
		= &  \langle a_{i_{k-1}}, x\km\rangle - b_{i_{k-1}} - \big(\langle a_{i_{k-1}}, x\km\rangle - b_{i_{k-1}}\big)\\
		=&  0,
	\end{aligned}
	\]	
	which indicates there is at least one element in $ \mathcal{N}_k $. 
	Hence, by the definition of $ \Gamma_k $, we have
	\begin{equation}\label{gamma-est}
		\Gamma_k\leq \left\{ \begin{array}{ll}
			\nmfsq{A}, & k = 0, \\
			\gamma,  & k\geq1.
		\end{array}\right.
	\end{equation}	
Then by Theorem \ref{theo:basic}, we can arrive at this corollary.
\end{proof}

\begin{remark}\label{remark-xie-114-1}
	With a similar analysis $($By taking $\Gamma_k = \nmfsq{A}$$)$, we can arrive at a deterministic version for \eqref{equ:convforgrdintexp}, i.e.
	\begin{equation}\label{equ-1108-d}
	\nmsq{x\k-x\s}\leq \left( 1-\hf\left( \frac{1}{\gamma}\nmfsq{A}+1 \right) \frac{\simi^2(A)}{\nmfsq{A}} \right)^{k-1} \left(1-\frac{\simi^2(A)}{\nmfsq{A}}\right) \nmsq{x\zo-x\s}.		
	\end{equation}
	As $\gamma<\|A\|^2_F$, it holds that
	$$
	1- \frac{\simi^2(A)}{\gamma}<1-\hf\left( \frac{1}{\gamma}\nmfsq{A}+1\right).
	$$
Hence, the convergence factor of the iGRK is smaller than that of the GRK method. 
\end{remark}
\begin{remark}\label{remark-xie-p}
From the proof of Corollary \ref{main-cor}, we can observe that for $k\geq1$, $r\k_{i_{k-1}}=0$. This implies that at least $i_{k-1}$ belongs to the index set $\mathcal{N}_k$ for $k\geq1$. Therefore, in practical computations, we can take an approximate $\Gamma_k$ as
$$
\tilde{\Gamma}_k= \left\{ \begin{array}{ll}
		\nmfsq{A}, & k = 0, \\
		\nmfsq{A}-\|a_{i_{k-1}}\|^2_2,  & k\geq1.
	\end{array}\right.
$$
\end{remark}

\subsection{Iteration complexity}

This subsection aims to demonstrate that the deterministic convergence result for the GRK method can achieve a better iteration complexity than the one obtained from the expected error $\mathbb{E}[\|x^{(k)}-x\s\|^2_2]$. The following lemma is essential for this purpose.

\begin{lemma}[Theorem 1, \cite{richtarik2014iteration}]\label{lamma-compx}
 Fix $x^{(0)}\in\mathbb{R}^n$ and let $\{x\k \}_{k\geq0}$ be a sequence of random vectors in $\mathbb{R}^n$ with
$x\kp$ depending on $x\k$ only. Let $\phi: \mathbb{R}^n \to \mathbb{R}$ be a nonnegative function and define
$\xi_k = \phi\left(x\k \right)$. Lastly, choose accuracy level $0 <\varepsilon  < \xi_0$, confidence level $\rho\in (0, 1)$,
and assume that the sequence of random variables $\{\xi_k \}_{k\geq0}$ is nonincreasing and  $\mathbb{E}[\xi_{k+1} | x\k] \leq \left(1-\frac{1}{c}
\right)\xi_k$, for all $k$ such that $\xi_k \geq \varepsilon$, where $c> 1$ is a constant. If we choose
$$K \geq c \log \frac{\xi_0}{\varepsilon\rho}, $$
then
$$\prob\left(\xi_K \leq \varepsilon \right) \geq 1-\rho.$$
\end{lemma}
	
From the argument in the proof of \cite[Theorem 1]{bai2018greedy}, for any $k\geq0$, we can get the following inequality for the GRK method
$$	
\begin{aligned}
	\mathbb{E}[\|x\kp-x\s\|^2_2|x\k]&\leq\max\left\{1-\hf\left( \frac{1}{\gamma}\nmfsq{A}+1\right),1-\frac{\sigma_{\min}^2(A)}{\|A\|^2_F}
	\right\}\|x\k-x\s\|^2_2\\
	&=\left(1-\frac{\sigma_{\min}^2(A)}{\|A\|^2_F}\right)\|x\k-x\s\|^2_2.
\end{aligned}
$$
Let $\xi_{k+1}=\|x\kp-x\s\|^2_2$, $0<\varepsilon<\|x^{(0)}-x\s\|^2_2$, and $\rho\in(0,1)$ be chosen arbitrarily. Then, using Lemma \ref{lamma-compx}, we can conclude that for all
\begin{equation}\label{IC-R}
	k\geq K_1:=\frac{\|A\|^2_F}{\sigma_{\min}^2(A)} \log \frac{\|x^{(0)}-x\s\|^2_2}{\varepsilon\rho},
\end{equation}
it holds
$$
\prob\left(\|x\k-x\s\|^2_2\leq\varepsilon\right)\geq 1-\rho.
$$	
This indicates that if we only have the linear convergence of the expected norm of the error, i.e. $\mathbb{E}[\|x\k-x\s\|^2_2]$, for the GRK method, it has an iterative complexity of the form  \eqref{IC-R}.


On the other hand, from \eqref{equ-1108-d}, we know that the iteration sequence of the GRK method satisfies
$$
\begin{aligned}
\|x\k-x\s\|^2_2\leq\left(1-\frac{\sigma_{\min}^2(A)}{\|A\|^2_F}\right)^k\|x^{(0)}-x\s\|^2_2.
\end{aligned}
$$
Since  $1-t<e^{-t}$ for any $t\in(0,1)$, we can simplify the inequality as 
	$$
	\|x\k-x\s\|^2_2\leq e^{-k\frac{\sigma_{\min}^2(A)}{\|A\|^2_F}}\|x^{(0)}-x\s\|^2_2.
	$$
	Hence,
	 for all 
	\begin{equation}\label{IC-D}
		k\geq K_2:=\frac{\|A\|^2_F}{\sigma_{\min}^2(A)} \log \frac{\|x^{(0)}-x\s\|^2_2}{\varepsilon},
	\end{equation}
	it holds
	$$
	\|x\k-x\s\|^2_2\leq\varepsilon.
	$$	
	This means that our deterministic linear convergence result ensures the GRK method has an iterative complexity of the form \eqref{IC-D}. It is obvious that the iteration complexity \eqref{IC-D} is better than \eqref{IC-R}.
	
\section{Momentum acceleration}


In this section, we will incorporate Polyak's heavy ball momentum into the iGRK method and show that the momentum variant of the GRK method also achieves deterministic linear convergence.

\subsection{Heavy ball momentum}
Recall that the gradient descent (GD) method for solving the optimization problem
$
\min\limits_{x\in\mathbb{R}^n} f(x)
$ utilizes the update
$$
x\kp=x\k-\alpha_{k} \nabla f\big(x\k\big),
$$
where $\alpha_{k}>0$ is the step-size, $f$ is a differentiable convex function, and $\nabla f\left(x\k\right)$ denotes the gradient of $f$ at $x\k$.
When $f$ is a convex function with $L$-Lipschitz gradient, GD requires $O(L / \varepsilon)$ steps to guarantee an error within $\varepsilon$.  If $f$ is also $\mu$-strongly convex, it converges linearly with a iteration complexity  of $O(\log (\varepsilon^{-1})(L / \mu))$ \cite{nesterov2003introductory}.
To improve the convergence behavior of the method, Polyak modified GD by introducing a momentum term, $\beta_k(x\k-x\km)$. This leads to the gradient descent method with momentum (mGD), commonly known as the heavy ball method
$$
x\kp=x\k-\alpha_{k} \nabla f\big(x\k\big)+\beta_k\big(x\k-x\km\big) .
$$
Polyak \cite{polyak1964some} proved that, for twice continuously differentiable objective functions $f(x)$ with $\mu$-strongly convex and $L$-Lipschitz gradient, mGD achieves a local accelerated iteration complexity  of  $O\left(\log (\varepsilon^{-1})\sqrt{L / \mu}\right)$ (with an appropriate choice of the step-size $\alpha_{k}$ and momentum parameter $\beta_k$). In this section, we aim to use the heavy ball momentum technique to  improve the performance of the GRK method.
	
\subsection{The proposed method}
The proposed mGRK method for solving linear systems utilizes the following update rule:
\begin{equation}\label{equ:algoiter}
x\kp = x\k - \al \frac{\langle a\ik, x\k\rangle - b\ik}{\nmsq{a\ik}} a\ik +\be(x\k-x\km),
\end{equation}
where the index $i_k$ is selected using a certain greedy probability criterion.
We note that Bai and Wu \cite{bai2018relaxed} introduced a relaxation parameter $\theta$  in the GRK method such that the factor $\frac{1}{2}$ before the terms $ \max_{i\in[m]} \left\{ \frac{|\langle a_i, x\k\rangle- b\i|^2}{\nmsq{a\i}} \right\} $ and $ \frac{\nmsq{Ax\k-b}}{\|A\|^2_F} $ is replaced by $\theta$ in the first term and by $1-\theta$ in the second term, proposing the relaxed greedy probability criterion. The  mGRK method will adopt this relaxed greedy probability criterion and  is described in Algorithm \ref{algo:main}.

\begin{algorithm}[htpb]
\caption{Improved greedy randomized Kaczmarz method with momentum (mGRK) \label{algo:main}}
\begin{algorithmic}
\Require
 $A\in \mathbb{R}^{m\times n}$, $b\in \mathbb{R}^m$, $\alpha>0$, $\beta\geq0$, $ \theta \in [0,1] $, $k=1$ and initial points $x\zo=x\ol \in\mathbb{R}^n$.
\begin{enumerate}
\item[1:] Determine the set
\[
\mathcal{N}_k = \{ i\in[m]: |\inn{a_i,x\k}-b_i|\neq0 \}
\]
and compute \[
	\Gamma_k = \sum_{i\in\mathcal{N}_k} \nmsq{a_i}.
\]
\item[2:] Determine the index set of positive integers
\begin{equation}\label{equ:algo2Ik}
	\Sk = \left\{  i : \frac{|\langle a_i, x\k\rangle- b\i|^2}{\nmsq{a\i}}  \geq \theta  \max_{i\in[m]} \left\{ \frac{|\langle a_i, x\k\rangle- b\i|^2}{\nmsq{a\i}} \right\} +  (1 - \theta) \frac{\nmsq{Ax\k-b}}{\Gamma_k}    \right\}.
\end{equation}
\item[3:] Select $ i_k \in  \Sk $ according to some probability
$$
			\prob(i_k = i) = p\k_i,
$$			
		where $ p\k_i = 0 \operatorname{if} i \notin \Sk $, $ p\k_i \geq 0 \operatorname{if} i \in \Sk $, and $ \sum_{i\in\Sk} p\k_i=1 $.
\item[4:] Set
$$
x\kp = x\k - \al \frac{\langle a\ik, x\k\rangle- b\ik}{\nmsq{a\ik}}a\ik+\be(x\k-x\km).
$$
\item[5:] If the stopping rule is satisfied, stop and go to output. Otherwise, set $k=k+1$ and return to Step $1$.
\end{enumerate}

\Ensure
  The approximate solution $x^k$.
\end{algorithmic}
\end{algorithm}

\subsection{Convergence analysis}
To establish the linear convergence of mGRK, the following lemma is instrumental.
	\begin{lemma}[\cite{han2022pseudoinverse}, Lemma 8.1]
\label{lemma:mom}
Fix $F^{(1)}=F^{(0)}\geq 0$ and let $\{F^{(k)}\}_{k\geq 0}$ be a sequence of nonnegative real numbers satisfying the relation
$$
F^{(k+1)}\leq \gamma_1 F^{(k)}+\gamma_2F^{(k-1)},\ \ \forall \ k\geq 1,
$$
where $\gamma_2\geq0,\gamma_1+\gamma_2<1$. Then the sequence satisfies the relation
$$F^{(k+1)}\leq q^k(1+\delta)F^{(0)},\ \ \forall \ k\geq 0,$$
where $$q=\left\{
           \begin{array}{ll}
             \frac{\gamma_1+\sqrt{\gamma_1^2+4\gamma_2}}{2}, & \hbox{if $\gamma_2>0$;} \\
             \gamma_1, & \hbox{if $\gamma_2=0$,}
           \end{array}
         \right. \  \mbox{and} \ \delta=q-\gamma_1\geq 0.$$ Moreover,
$
\gamma_1+\gamma_2\leq q<1,
$
with equality if and only if $\gamma_2=0$.
	\end{lemma}
	
We have the following convergence result for Algorithm \ref{algo:main}.
	\begin{theorem}\label{theo:main}
Suppose that $ x\zo=x\ol \in \Rn $, $ \theta \in [0,1] $ and let $ x\s = A^\dagger b + (I-A^\dagger A)x\zo $ denote  the projection of $ x\zo $ onto the solution set of $ Ax=b. $
Assume that $ \al \in (0,2) $ if $\beta=0$ or $\al\in(0,1+\beta)$ if $ \be >0 $, and  the expressions
$$\gamma_1 = 2\be^2 + 3\be+ 1 - (3\al\be + 2\al - \al^2)\frac{\simi^2(A)}{\nmfsq{A}} \quad  \mathrm{and} \quad \gamma_2 = 2\be^2 + \be$$
satisfy $\gamma_1+\gamma_2<1$. Then the iteration sequence $ \{x^{(k)}\}_{k\geq 0} $ generated by Algorithm \ref{algo:main} satisfies
$$
\nmsq{x\kp-x\s} \leq q^k(1+\delta)\nmsq{x\zo-x\s},
$$
where $ q = \frac{\gamma_1+\sqrt{\gamma_1^2+4\gamma_2}}{2} $ and $ \delta = q -\gamma_1. $ Moreover, $ \gamma_1+\gamma_2\leq q <1 $.
\end{theorem}
	\begin{proof}
To state conveniently, we set $ P\ik := \frac{a\ik a\ik\t}{\nmsq{a\ik}} $, then we have
\begin{equation}\label{equ:deft}
\frac{\langle a\ik, x\k\rangle - b\ik}{\nmsq{a\ik}}a\ik = \frac{a\ik a\ik\t}{\nmsq{a\ik}} (x\k-x\s) = P\ik(x\k -x\s),
\end{equation}
where the first equality follows from the fact $ a\i\t x\s = b\i $ for any $ i\in \setm $.
		Noting that $  P\ik \t = P\ik $ and $ P\ik^2 = \frac{a\ik a\ik\t a\ik a\ik\t}{\nm{a\ik}^4} = \frac{a\ik a\ik\t}{\nmsq{a\ik}} = P\ik $, we have
\begin{equation}\label{equ:t2=t}
\inn{x\k-x\s, P\ik(x\k-x\s)} = \inn{x\k-x\s, P\ik^2(x\k-x\s)} = \nmsq{P\ik(x\k-x\s)}.
\end{equation}
Now by the  iterative strategy of Algorithm \ref{algo:main}, we have
\begin{equation}\label{equ:mainframe}
\begin{aligned}
\nmsq{x\kp - x\s} = & \ \nmsq{x\k - x\s - \al P\ik(x\k - x\s) + \be(x\k-x\km)} \\
= & \ \nmsq{(I-\al P\ik)(x\k - x\s)} + \be^2 \nmsq{x\k-x\km} \\
& + 2\be \inn{x\k - x\s, x\k-x\km} -2\al\be\inn{P\ik(x\k - x\s),x\k-x\km}.
\end{aligned}
\end{equation}
We shall analyze the four terms in the last expression  separately. By using \eqref{equ:t2=t}, the first term satisfies
		\[\begin{aligned}
			\nmsq{(I-\al P\ik)(x\k - x\s)} = & \ \inn{(x\k - x\s),(I-2\al P\ik+\al^2P\ik^2)(x\k - x\s)} \\
			= & \ \nmsq{x\k - x\s} - (2\al-\al^2) \nmsq{P\ik(x\k-x\s)}.
		\end{aligned}\]
		We keep the second term $ \be^2 \nmsq{x\k-x\km} $ unchanged and reformulate the third term by
		\[ \begin{aligned}
			2\be \inn{x\k - x\s, x\k-x\km} = \be (\nmsq{x\k - x\s} + \nmsq{x\k-x\km} - \nmsq{x\km - x\s} )
		\end{aligned} \]
For the last term,
		\[ \begin{aligned}
			& -2\al\be\inn{P\ik(x\k - x\s),x\k-x\km} \\
			= & \ \al\be( \nmsq{ x\k-x\km - P\ik(x\k - x\s)} - \nmsq{x\k-x\km} - \nmsq{P\ik(x\k - x\s)} ) \\
			= & \ \al\be ( \nmsq{ (I-P\ik)(x\k-x\s) - (x\km-x\s) } - \nmsq{x\k-x\km} - \nmsq{P\ik(x\k - x\s)} ) \\
			\leq & \  \al\be ( 2\nmsq{ (I-P\ik)(x\k-x\s)} + 2\nmsq{ x\km-x\s } - \nmsq{x\k-x\km} - \nmsq{P\ik(x\k - x\s)} ) \\
			= & \ \al\be (2\nmsq{x\k-x\s} - 3\nmsq{P\ik(x\k-x\s)} + 2\nmsq{ x\km-x\s } - \nmsq{x\k-x\km}),
		\end{aligned} \]
where the inequality follows from $ \nmsq{a+b} \leq 2\nmsq{a} + 2\nmsq{b} $. Overall, subsituting the above bounds into \eqref{equ:mainframe}, we obtain
		\[\begin{aligned}
			\nmsq{x\kp - x\s} \leq & \ (2\al\be+\be+1) \nmsq{x\k-x\s} + (2\al\be-\be)\nmsq{x\km-x\s} \\
			&  +(\be^2 + \be - \al\be) \nmsq{x\k-x\km} - (3\al\be + 2\al - \al^2)\nmsq{P\ik(x\k -x\s)}.
		\end{aligned}\]
Since $\be^2 + \be - \al\be\geq0$ by the assumption in this theorem, we eliminate the term $ \nmsq{x\k-x\km} $ by $ \nmsq{x\k-x\km} \leq 2\nmsq{x\k-x\s} + 2\nmsq{x\km-x\s}$ and have
		\begin{equation}\label{equ:mainframe2}
			\begin{aligned}
				\nmsq{x\kp - x\s} \leq  & \ (2\be^2 + 3\be+ 1) \nmsq{x\k - x\s} + (2\be^2 + \be) \nmsq{x\km - x\s} \\
				& - (3\al\be + 2\al - \al^2)\nmsq{P\ik(x\k -x\s)}.
			\end{aligned}
		\end{equation}
Now we focus on the last term $ \nmsq{P\ik(x\k -x\s)} $ and establish its relationship with $ \nmsq{x\k-x\s}. $ It follows from \eqref{equ:algo2Ik}, we know that 
\[
\begin{aligned}
	 \nmsq{P\ik(x\k -x\s)} =& \frac{|\langle a\ik, x\k\rangle - b\ik|^2}{\nmsq{a\ik}} \\
	 \geq & \theta \max_{i\in[m]} \left\{ \frac{|a\i\t x\k- b\i|^2}{\nmsq{a\i}} \right\} + (1-\theta) \frac{\nmsq{Ax\k-b}}{\Gamma_k} \\
	 \geq & \frac{\nmsq{Ax\k-b}}{\Gamma_k},
\end{aligned}
\]
where the equality follows from \eqref{equ:deft}.
We now prove that $ x\k - x\s \in \operatorname{Range}(A^\top) $ for all $ k\geq 0 $ by induction. By the definition of $x\zo, x\ol$, and $x\s$, we have $ x\zo - x\s,x\ol - x\s\in \operatorname{Range}(A^\top) $. If $ x^\ell - x\s\in \operatorname{Range}(A^\top) $ holds for $\ell=0,\ldots,k$, then
$$ 
\begin{aligned}x\kp- x\s  &= x\k- x\s  - \al \frac{\langle a\ik, x\k\rangle - b\ik}{\nmsq{a\ik}}a\ik+\be(x\k-x\km)\\
 & =(1+\beta)(x\k- x\s) -\beta(x\km -x\s) - \al \frac{\langle a\ik, x\k\rangle - b\ik}{\nmsq{a\ik}}a\ik \in\operatorname{Range}(A^\top).
\end{aligned}$$ 
Hence, by induction we have that $ x\k - x\s \in \operatorname{Range}(A^\top) $ for all $ k\geq 0 $. Hence, we have
		\[\nmsq{P\ik(x\k -x\s)} \geq \frac{\nmsq{Ax\k-b}}{\Gamma_k}\geq \frac{\simi^2(A)}{\Gamma_k}\nmsq{x\k-x\s}\geq \frac{\simi^2(A)}{\nmfsq{A}}\nmsq{x\k-x\s}. \]
Substituting it into \eqref{equ:mainframe2}, we can get
$$
\begin{aligned}
\nmsq{x\kp - x\s} \leq  & \left(2\be^2 + 3\be+ 1 - (3\al\be + 2\al - \al^2)\frac{\simi^2(A)}{\nmfsq{A}} \right) \nmsq{x\k - x\s}
\\&+ (2\be^2 + \be) \nmsq{x\km - x\s}.
\end{aligned}
$$
Suppose that $ F^{(k)}:= \nmsq{x\k-x\s} $,
$\gamma_1 = 2\be^2 + 3\be+ 1 - (3\al\be + 2\al - \al^2)\frac{\simi^2(A)}{\nmfsq{A}}$, and $ \gamma_2  = 2\be^2 + \be$. Noting that the conditions of the Lemma \ref{lemma:mom} are satisfied. Indeed, $\gamma_2\geq0$, and if $\gamma_2=0$, then $\beta=0$ and $q=\gamma_1\geq0$.
The condition $\gamma_1+\gamma_2<1$ holds by assumption.
Then by Lemma \ref{lemma:mom}, one can get the theorem.
	\end{proof}
	
\begin{remark}\label{remark-xie-0624}
When $ \be = 0 $, the conclusion in Theorem \ref{theo:main} reduces to
\[
\nmsq{x\k-x\s} \leq \bigg( 1-(2\al-\al^2)\frac{\simi^2(A)}{\nmfsq{A}} \bigg)^k \nmsq{x\zo-x\s},
\]
which is the convergence rate for the GRK method with relaxation. It can be observed that when $ \be = 0 $ and $ \al = 1 $, the bound established in Theorem \ref{theo:basic} is tighter than the one obtained here.
	\end{remark}
	
\begin{remark}
In \cite[Theorem 1]{loizou2020momentum}, the authors provided a linear convergence result for the momentum variant of the RK method of the same form as ours, where
\[ \tilde{\gamma}_1 = 2\be^2 + 3\be+ 1 - (\al\be + 2\al - \al^2)\frac{\simi^2(A)}{\nmfsq{A}} \quad and \quad \tilde{\gamma}_2  = 2\be^2 + \be + \al\be\frac{\sima^2(A)}{\nmfsq{A}} . \]
These constants are larger than those obtained by Theorem \ref{theo:main}, and hence, a smaller convergence factor can be guaranteed for the mGRK method.
%
\end{remark}

\begin{remark}\label{remark-xie-114}
	From the proof of Theorem \ref{theo:main}, it is evident that one can set $\Gamma_k=\|A\|^2_F$ in Algorithm \ref{algo:main}. In this particular scenario, we can derive the same conclusion as stated in Theorem  \ref{theo:main}.
\end{remark}
	
Let us explain how to  choose the parameters $\alpha$ and $\beta$ such that $\gamma_1+\gamma_2<1$ is satisfied in Theorem \ref{theo:main}. Indeed, set
$$
\tau_1:=4-3\alpha\frac{\sigma_{\min}^2(A)}{\|A\|^2_F}
\ \text{and} \
\tau_2:=(2\alpha-\alpha^2)\frac{\sigma_{\min}^2(A)}{\|A\|^2_F}.
$$
Let $\alpha\in(0,1]$, then we have $\tau_2>0$ and the condition $\gamma_1+\gamma_2<1$ now is satisfied for all
\begin{equation}\label{xie-equ-0626}
0\leq \beta<\frac{1}{8}\bigg(\sqrt{\tau_1^2+16\tau_2}-\tau_1\bigg).
\end{equation}

Finally, let us compare the convergence rates  obtained in Theorem \ref{theo:main} and  Remark \ref{remark-xie-0624}.
From the definition of $\gamma_1$ and $\gamma_2$, we know that convergence rate $q(\beta)$  in Theorem \ref{theo:main} can be viewed as a function of $\beta$. We further assume that $0<\alpha<\min\left\{2,\frac{4\|A\|^2_F}{3\sigma^2_{\min}(A)}\right\}$ so that $\tau_1\geq0$. We note that $\frac{4\|A\|^2_F}{3\sigma^2_{\min}(A)}\geq2$ can be easily satisfied in practice, for instance, when $\mathrm{Rank}(A)\geq2$. Then we have
$$
q(\beta)\geq\gamma_1+\gamma_2=4\beta^2+\tau_1\beta-\tau_2+1\geq 1-\tau_2
=1-(2\alpha-\alpha^2)\frac{\sigma_{\min}^2(A)}{\|A\|^2_F}=q(0).
$$
Since the lower bound $q$ is an increasing function of $\beta$, we know  that the convergence rate for mGRK is always inferior to that of GRK. Although numerical experiments suggest that mGRK performs better than GRK in practice, it is challenging to achieve a better convergence rate in theory for mGRK.  One possible approach to overcome this problem is to use the adaptive strategy proposed in \cite{zeng2023adaptive}. 

	
	
\section{Numerical Experiments}

In this section, we present some numerical results for the mGRK method for solving linear systems. We also compare the mGRK method with the GRK method and the stochastic conjugate gradient (SCG) method \cite[Algorithm 4]{zeng2023adaptive} on a variety of test problems. Our numerical results indicate that incorporating momentum in the GRK method can lead to improved convergence and efficiency in solving linear systems.
 All methods are implemented in  {\sc Matlab} R2022a for Windows $10$ on a desktop PC with the  Intel(R) Core(TM) i7-10710U CPU @ 1.10GHz  and 16 GB memory.

We use the following two types of coefficient matrices. One is matrices randomly generated by the {\sc Matlab} function {\tt randn}. For given $m, n, r$, and $\kappa>1$, we construct a dense matrix $A$ by $A=U D V^\top$, where $U \in \mathbb{R}^{m \times r}, D \in \mathbb{R}^{r \times r}$, and $V \in \mathbb{R}^{n \times r}$. Using {\sc Matlab}  notation, these matrices are generated by {\tt [U,$\sim$]=qr(randn(m,r),0)}, {\tt [V,$\sim$]=qr(randn(n,r),0)}, and {\tt D=diag(1+($\kappa$-1).*rand(r,1))}. So the condition number and the rank of $A$ is upper bounded by $\kappa$ and $r$, respectively.
Another type data is the real-world data from SuiteSparse Matrix Collection \cite{Kol19}.

In our implementations, to ensure the consistency of the linear system, we first generate the solution by $x={\tt randn(n,1)}$ and then set $b=Ax$. All computations are started from the initial vector $x^{(0)}=0$.  We stop the algorithms if the relative solution error (RSE)
$\frac{\|x\k-A^\dagger b\|^2_2}{\|A^\dagger b\|^2_2}\leq10^{-12}$. All the results are averaged over $20$ trials and we report the average number of iterations (denoted as Iter) and the average computing time in seconds (denoted as CPU). We set the step-size $\alpha=1$, the relaxation parameter $\theta=\frac{1}{2}$, and $\Gamma_k=\|A\|^2_F$ (see Remark \ref{remark-xie-114}) for the mGRK method. We use \eqref{xie-230624} as the  probability criterion for selecting the working row. For the SCG method  \cite[Algorithm 4]{zeng2023adaptive}, at the $k$-th step, the index $ i_k$ is selected randomly with probability $\prob(i_k = i) = \frac{\|a\i\|^2_2}{\|A\|^2_F}$.

Figure \ref{figue1p} illustrates a comparison of the performance between the iGRK method and the GRK method. For the iGRK method, we set $\Gamma_0= \nmfsq{A}$ and $\Gamma_k=\nmfsq{A}-\|a_{i_{k-1}}\|^2_2$ for $k\geq1$, as mentioned in Remark \ref{remark-xie-p}. As depicted in Figure \ref{figue1p}, the iGRK method exhibits slightly improved performance compared to the GRK method. Based on this observation,  our subsequent test will focus solely  on comparing   the GRK method and the mGRK method.

\begin{figure}[hptb]
	\centering
	\begin{tabular}{cc}
		\includegraphics[width=0.4\linewidth]{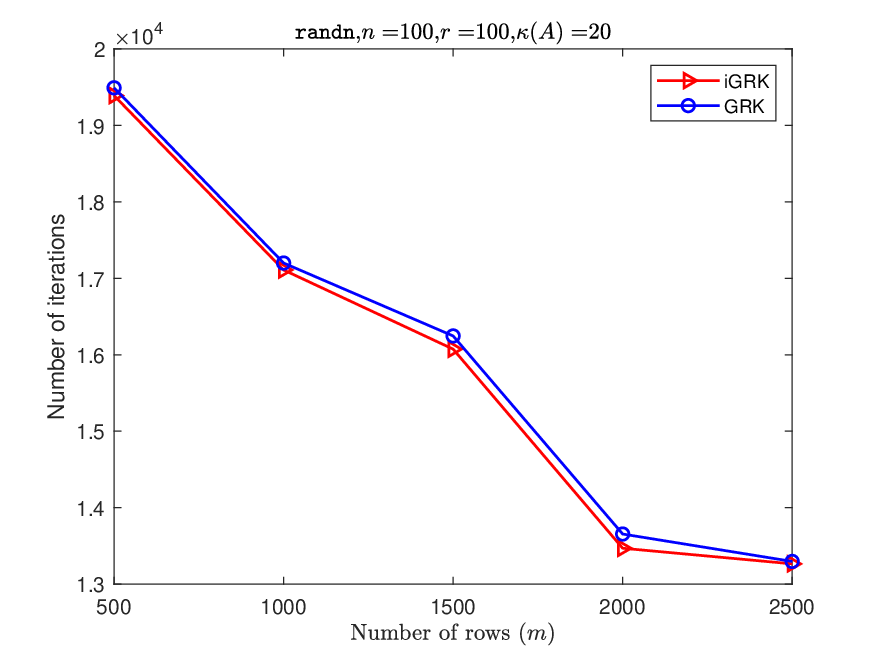}
		\includegraphics[width=0.4\linewidth]{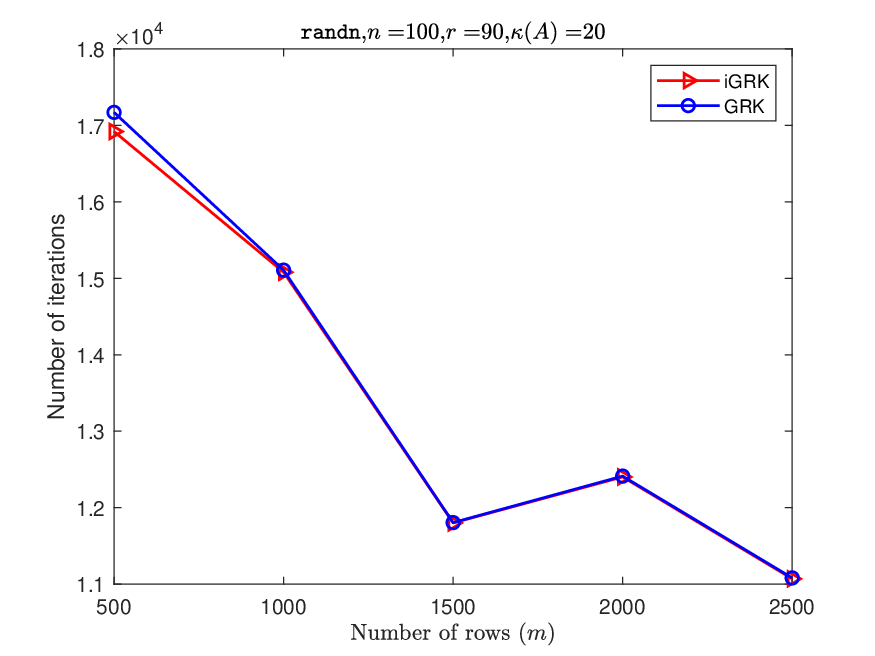}
	\end{tabular}
	\caption{Figures depict the number of iterations vs increasing number of rows for the case of the random matrix. The title of each plot indicates the values of $n,r$, and $\kappa$. }
	\label{figue1p}
\end{figure}

Figures \ref{figue1} and \ref{figue2} illustrate our experimental results with different choices of the momentum parameter $\beta$. When $\beta=0$, the mGRK method is exactly the GRK method.
We note that in all of the presented tests, the momentum parameters $\beta$ of the methods are chosen as non-negative constants that do not depend on parameters not known to the users, such as $\sigma^2_{\min}(A)$.  It is evident that the incorporation of the momentum term has resulted in an improvement in the performance of the GRK method. We can observe that $\beta=0.4$ is consistently a good choice for achieving sufficiently fast convergence of the mGRK method on random matrices. While for the datasets {\tt WorldCities} and {\tt crew1}, we find that $\beta=0.6$ is a good option for {\tt WorldCities} (we have not plotted the case where $\beta=0.7$ as the mGRK method is observed to be divergent in this case) and $\beta=0.4$ is a good option for {\tt crew1}. This indicates that we need to select appropriate values of $\beta$ for handing different types of data.

\begin{figure}[hptb]
	\centering
	\begin{tabular}{cc}
\includegraphics[width=0.4\linewidth]{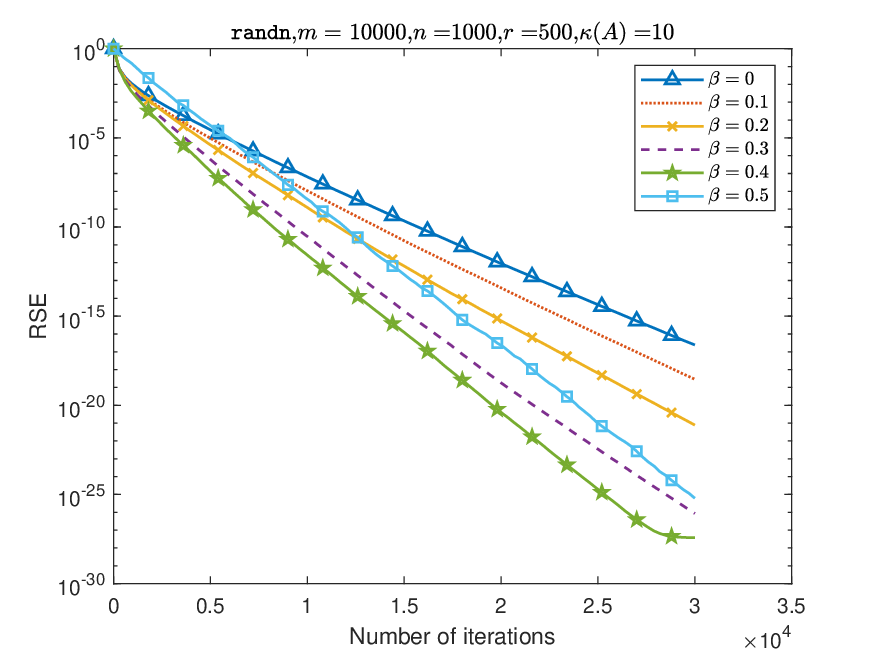}
\includegraphics[width=0.4\linewidth]{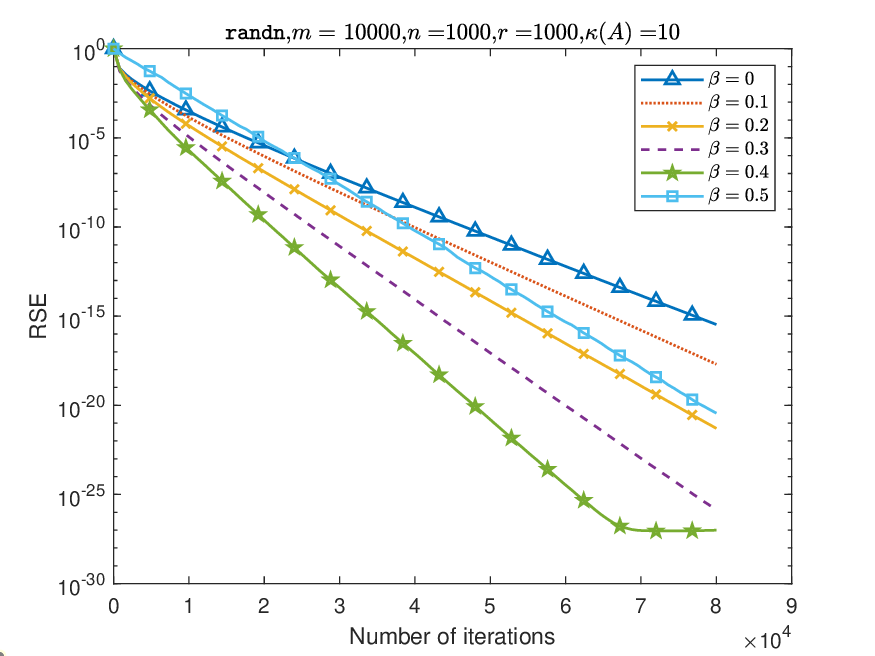}
	\end{tabular}
	\caption{Performance of mGRK with different momentum parameters $\beta$. The coefficient matrices are random matrices and the title of each plot indicates the values of $m,n,r$, and $\kappa$. }
	\label{figue1}
\end{figure}

\begin{figure}[hptb]
	\centering
	\begin{tabular}{cc}
\includegraphics[width=0.4\linewidth]{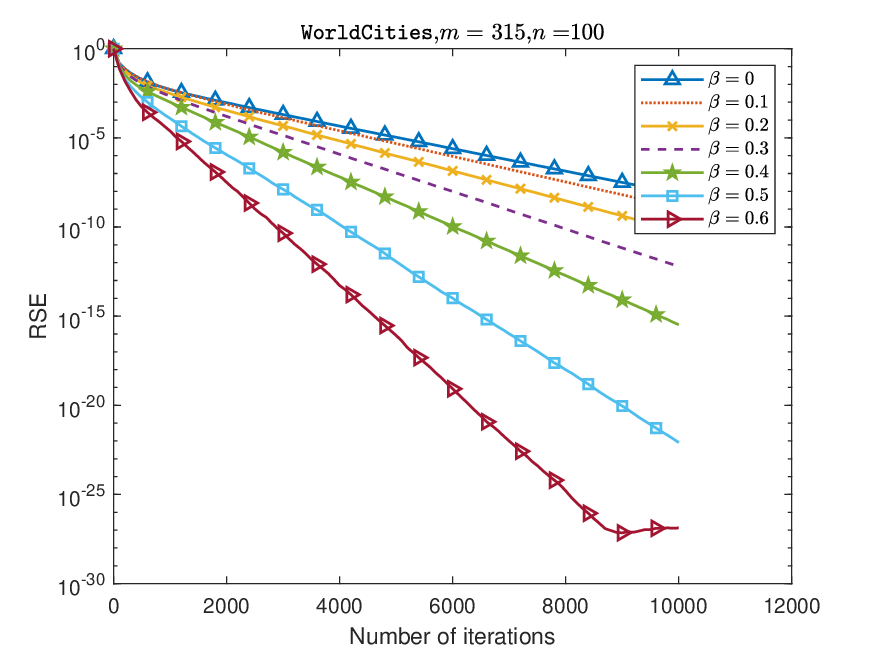}
\includegraphics[width=0.4\linewidth]{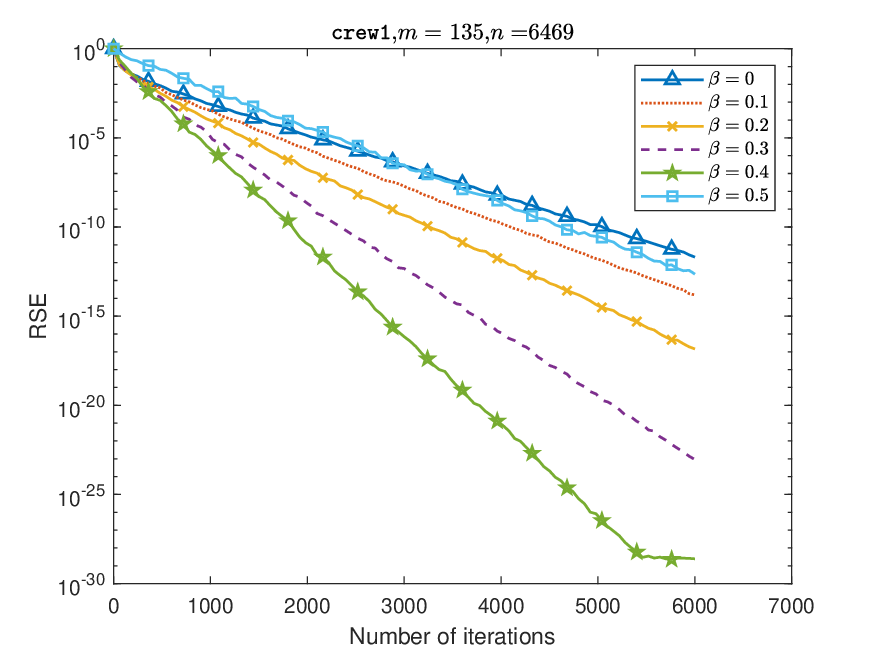}
	\end{tabular}
	\caption{Performance of mGRK with different momentum parameters $\beta$ with coefficient matrices from SuiteSparse Matrix Collection \cite{Kol19}.  }
	\label{figue2}
\end{figure}

In Figure \ref{figue3}, we  present the computing time of GRK, SCG, and mGRK with random matrices $A$, where $m=1000,2000,\ldots,10000,n=100$, $\kappa(A)=10$ (top) or $\kappa(A)=40$ (below), and $r=100$ (left) or $r=90$ (right). From the figure, we can observe that the GRK method and the SCG method exhibit comparable performance. Particularly, if the condition number of the coefficient matrix is larger, the SCG method may outperform the GRK method. Additionally, regardless of whether the coefficient matrix $A\in\mathbb{R}^{m\times n}$ is full rank or rank-deficient, the mGRK method outperforms both the GRK method and the SCG method in terms of CPU time. Specifically, the mGRK method is approximately two times faster than the GRK method.

\begin{figure}[hptb]
	\centering
	\begin{tabular}{cc}
\includegraphics[width=0.4\linewidth]{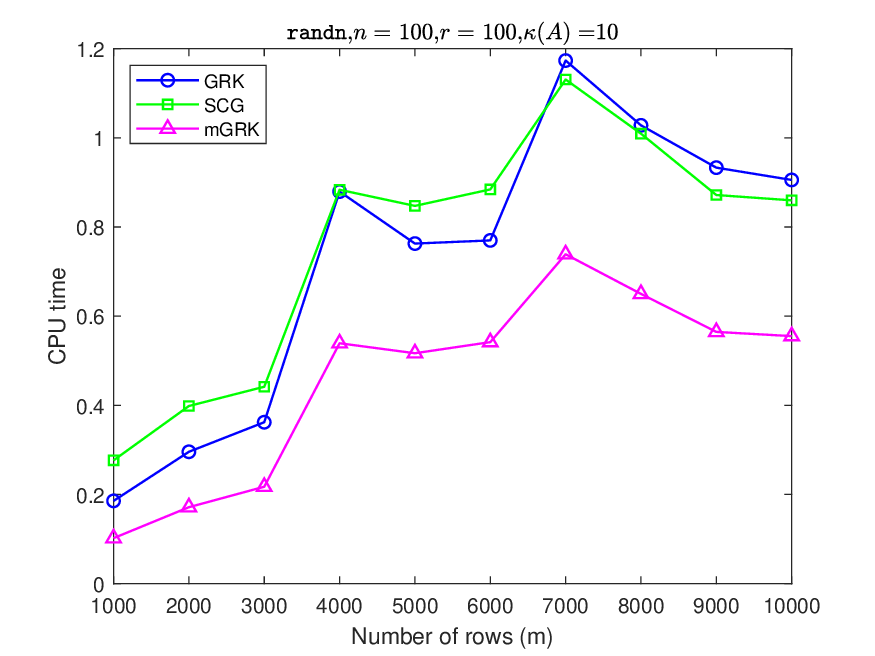}
\includegraphics[width=0.4\linewidth]{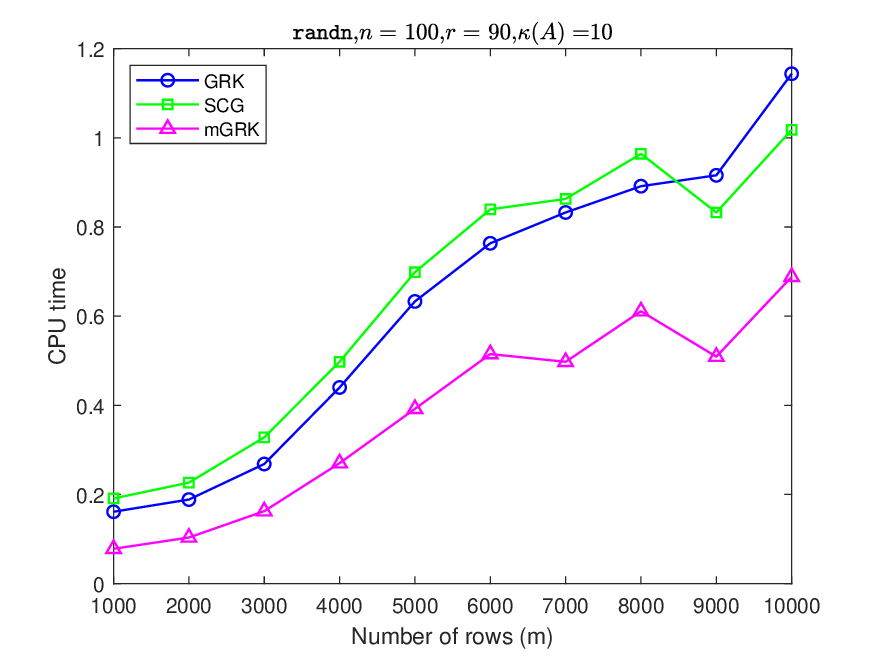}\\
\includegraphics[width=0.4\linewidth]{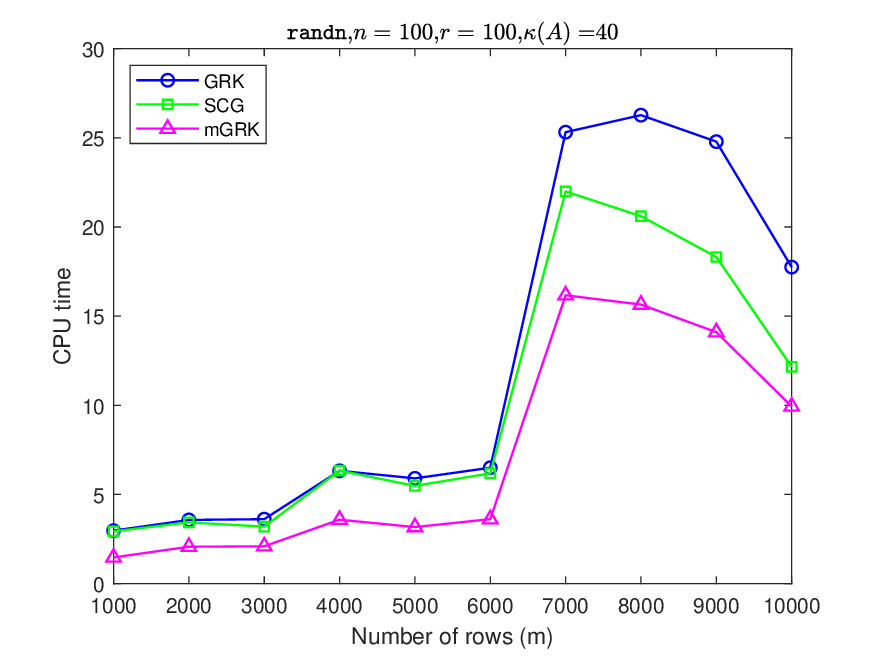}
\includegraphics[width=0.4\linewidth]{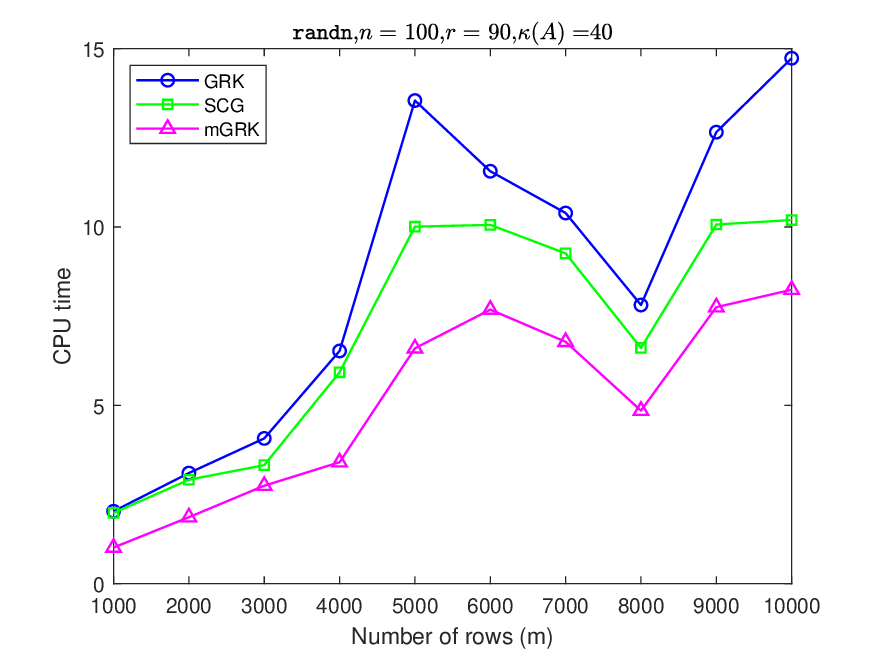}
	\end{tabular}
	\caption{Figures depict the CPU time (in seconds) vs increasing number of rows for the case of the random  matrix.  The title of each plot indicates the values of $n,r$, and $\kappa$. We set $\beta=0.4$ for the mGRK method.}
	\label{figue3}
\end{figure}

Table \ref{table1} presents the iteration counts and computing times for the GRK, SCG, and mGRK methods when applied to sparse matrices obtained from the SuiteSparse Matrix Collection \cite{Kol19}. The matrices used in the experiments include {\tt bibd\_16\_8}, {\tt crew1}, {\tt WorldCities}, {\tt nemsafm}, {\tt model1}, {\tt ash958}, {\tt Franz1}, and {\tt mk10-b2}. Some of these matrices are full rank, while others are rank-deficient. From Table \ref{table1}, it can be observed that both the GRK and mGRK methods are more effective than the SCG method. This is because, in the case of sparse matrices, the computation of the parameters $\varepsilon_k$, i.e., $Ax^{(k)}-b$, in both GRK and mGRK is significantly reduced. It can be also observed that the mGRK method with an appropriate momentum parameter generally exhibits better performance  than the GRK method.

\begin{table}
\renewcommand\arraystretch{1.5}
\setlength{\tabcolsep}{2pt}
\caption{ The average Iter and CPU of GRK, SCG, and mGRK for linear systems with coefficient matrices from SuiteSparse Matrix Collection \cite{Kol19}. The appropriate momentum parameters $\beta$ for mGRK are also listed.}
\label{table1}
\centering
{\scriptsize
\begin{tabular}{  |c| c| c| c| c |c |c |c| c|c| c| }
\hline
\multirow{2}{*}{ Matrix}& \multirow{2}{*}{ $m\times n$ }  &\multirow{2}{*}{rank}& \multirow{2}{*}{$\frac{\sigma_{\max}(A)}{\sigma_{\min}(A)}$}  &\multicolumn{2}{c| }{GRK}  &\multicolumn{2}{c| }{SCG} &\multicolumn{3}{c| }{mGRK}
\\
\cline{5-11}
& &   &    & Iter & CPU    & Iter & CPU  & Iter & CPU &$\beta$    \\
\hline
{\tt bibd\_16\_8}& $120\times12870$ &  120  & 9.54 &     2191.60  &   4.9351 &  5985.90  &   9.1321 &   983.40  &   {\bf 2.1297}  & 0.4 \\
\hline
{\tt crew1} & $135\times6469 $ &  135  &18.20  &   6068.40  &   3.3238 & 29097.90  &   9.3263 &  2099.40  &   {\bf 1.0639} &0.4\\
\hline
{\tt WorldCities} & $315\times100$ &  100  &6.60  &   14447.20  &   1.2846 & 63821.40  &   1.0471 &  6418.30  &   {\bf 0.5549}  & 0.6  \\
\hline
{\tt nemsafm} & $334\times 2348$ &   334  &4.77  &    2555.40  &   0.2782 & 20033.90  &   1.4815 &  2052.40  &   {\bf 0.2097} &0.1 \\
\hline
{\tt model1} & $  362\times798 $ &  362  & 17.57 &11608.20  &   0.8137 & 122107.60  &   4.0335 &  5609.20  &   {\bf0.3951} &0.3\\
\hline
{\tt ash958} & $958\times292$ &  292  &3.20  &  1619.60  &   0.0931 & 11480.30  &   0.2067 &  1448.10  &  {\bf 0.0786}  & 0.1\\
\hline
{\tt Franz1} & $ 2240\times768 $ &  755  & 2.74e+15 &  15240.20  &   1.3315 & 68401.80  &   2.3409 &  5748.20  &   {\bf0.5248} &0.3 \\
\hline
{\tt mk10-b2} & $  3150\times630 $ &  586 & 2.74e+15 &2326.20  &   0.2386 & 16447.00  &   0.7240 &  2285.30  &   {\bf0.2086}  &0.1\\
\hline
\end{tabular}
}
\end{table}


\section{Concluding remarks}
In this paper, we proved that the linear convergence of the greedy randomized Kaczmarz (GRK) method is deterministic. Moreover, we showed that the deterministic convergence can be inherited to the heavy ball momentum variant of the GRK method. The preliminary numerical results showed that the mGRK method performs better than the GRK method.

It can be seen from \eqref{xie-equ-0626} that, theoretically, the optimal choices of $\alpha$ and $\beta$ for the mGRK method require knowledge of the smallest nonzero singular value of matrix, which is often not  accessible.
Furthermore, during our experiments, we have observed the need to select suitable parameters $\beta$ for different types of data. Hence, it would  be a valuable topic to investigate the GRK method with adaptive heavy ball momentum \cite{zeng2023adaptive}, where the parameters can be learned adaptively using iterative information. This adaptive approach could potentially overcome the challenge of selecting optimal parameters and improve the overall performance of the GRK method in various scenarios.
	
\bibliographystyle{plain}
\bibliography{references1}

\end{document}